\documentclass[11pt, lettersize, twoside]{article}
\usepackage{graphicx}
\usepackage{amssymb}
\usepackage{amsmath}
\usepackage{fancyhdr}

\newtheorem{thm}{Theorem}[section]

\newtheorem{prop}[thm]{Proposition}
\newtheorem{defn}[thm]{Definition}
\newtheorem{lem}[thm]{Lemma}

\newcounter{step}
\newcounter{claim}
\newcounter{subclaim}

\newenvironment{proof}{\textbf{Proof}\ }{\hfill $\square$}

\newenvironment{claim}{\vspace{0.75\baselineskip}
\textbf{Claim\ }}{\vspace{0.75\baselineskip}}

\newcommand*{\dateenglish}{\renewcommand*{\today}{%
	\number\day \ifcase\day \or
	st\or nd\or rd\or th\or th\or th\or th\or th\or th\or th\or
	th\or th\or th\or th\or th\or th\or th\or th\or th\or th\or
	st\or nd\or rd\or th\or th\or th\or th\or th\or th\or th\or
	st\fi\space \ifcase\month \or
	January\or February\or March\or April\or May\or June\or
	July\or August\or September\or October\or November\or
	December\fi \space\number\year}}

\newcommand{\oshs}{one-sided Heegaard splitting}
\newcommand{\oss}{one-sided splitting}
\newcommand{\tshs}{two-sided Heegaard splitting}
\newcommand{\tss}{two-sided splitting}
\newcommand{\gi}{geometrically incompressible}
\newcommand{\gc}{geometrically compressible}

\dateenglish

\begin{document}
\thispagestyle{plain}

\scalebox{2}{}

\begin{center}
\scalebox{1.5}{\textbf{Incompressible one-sided surfaces in}}

\scalebox{1.5}{\textbf{even fillings of Figure 8 knot space}}
\vspace{0.5\baselineskip}

\textsc{Loretta Bartolini}
\vspace{0.5\baselineskip}

\parbox{.8\textwidth}
{
\textbf{Abstract}
\vspace{.3\baselineskip}

\footnotesize{In the closed, non-Haken, hyperbolic class of examples generated by $(2p, q)$ Dehn fillings of Figure 8 knot space, the \gi{} one-sided surfaces are identified by the filling ratio $\frac{p}{q}$ and determined to be unique in all cases. When applied to \oshs{s}, this can be used to classify all \gi{} splittings in this class of closed, hyperbolic examples; no analogous classification exists for \tshs{s}.}
}
\end{center}
\vspace{\baselineskip}

\section{Introduction}

Knot and link exteriors provide a rich class of bounded manifolds, which form the basis of a wide range of $3$-manifold study. Many geometric techniques capitalise on the accessibility provided by knot and link diagrams, with particular reference to surfaces spanned by the link. Given any closed $3$-manifold can be realised by Dehn surgery on a link in $S^3$ \cite{lickorish,wallace}, one can hope to transfer the wealth of information about bounded surfaces in such exteriors to closed $3$-manifolds in general.

Whilst it is very straightforward to construct closed surfaces after Dehn filling, it is generally extremely difficult to determine the properties of such surfaces. In particular, the properties of incompressibility and strong irreducibility for two-sided surfaces are difficult both to predict and determine. Admitting non-orientability, however, gives rise to large classes of surfaces to which additional arguments specific to one-sided surfaces may be applied.

Under the restriction of orientability, a bounded two-sided surface in a link space can be `capped off' in a solid filling torus by the addition of a collection of meridian discs. As such, this process applies only to fillings along the boundary slope of a given orientable surface. In contrast, the admission of non-orientability for the closed surface allows a bounded surface in a link space, regardless of orientability, to be completed by attaching a \gi{}, one-sided surface in a solid filling torus. The only restriction required to ensure the existence of such a surface is that the filling has $(2p,q)$ slope. This is consistent with the requirement that a $3$-manifold have non-zero $\mathbb{Z}_2$-homology in order for it to contain closed, embedded, non-orientable surfaces.

Since infinite classes of fillings admit such surfaces, one can use information about surfaces in a single knot exterior to obtain information about large classes of closed examples. The Figure 8 knot is well-known and widely studied, with almost all Dehn fillings affording closed, hyperbolic, non-Haken $3$-manifolds. Using the known \gi{} surfaces in this exterior, one can use non-orientable surface techniques to classify the closed, \gi{}, one-sided surfaces in this interesting class of examples.

Many thanks to Hyam Rubinstein for helpful discussions throughout the preparation of this paper.

\section{Preliminaries}

Let $M$ be a compact, connected, orientable, irreducible $3$-manifold and let all maps be considered at $PL$.

\begin{defn}\label{defn:gi}
A surface $K \not= S^2$ embedded in $M$ is \gc{} if there exists an embedded, non-contractible loop on $K$ that bounds an embedded disc in $M$. Call $K$ \gi{} if it is not \gc{}.
\end{defn}

Since one-sided surfaces admit compressing discs with double points on the boundary, \gi{} surfaces normally have such singular discs.

\begin{defn}
A pair $(M, K)$ is called a \oshs{} if $K$ is a closed non-orientable surface embedded in a closed orientable $3$-manifold $M$ such that $H = M \setminus K$ is an open handlebody.
\end{defn}

Note that any \gi{} one-sided surface in an irreducible, non-Haken $3$-manifold provides a \oshs{}, since the boundary of the twisted $I$-bundle over the surface compresses to a collection 2-spheres, as discussed in Rubinstein~\cite{rubin78}. As such, within this class, \gi{} surfaces and \oss{s} are used interchangeably. The following statement relating existence to non-zero $\mathbb{Z}_2$-homology, therefore, applies equally to both perspectives:

\begin{thm}[\cite{rubin78}]\label{thm:incomp}
If $M$ is closed, orientable, irreducible and non-Haken, then there is a \gi{} \oss{} associated with any non-zero class in $H_2(M, \mathbb{Z}_2)$.
\end{thm}

In compact $3$-manifolds with non-empty boundary, the usual notion of boundary compressions generalises to one-sided surfaces:

\begin{defn}
A bounded surface $K \subset M$ with $\partial K \subset \partial M$ is boundary compressible if there exists an embedded bigon $B \subset M$ with $\partial B = \alpha \cup \beta$, such that $\alpha = B \cap K$ is essential, $\beta = B \cap \partial M$ and $\alpha \cap \beta = \partial \alpha = \partial \beta$.
\end{defn}

If $K$ is one-sided and $K \cap T_k$ is a single essential loop for some torus $T_k \subseteq \partial M$, a boundary compressing bigon can correspond to the boundary compression of a M\"obius band. In this case, $\beta$ has ends on locally opposite sides of $\partial K$; if $\beta$ has ends on locally the same side of $\partial K$, this corresponds to a boundary compression in the usual sense of orientable surfaces. Note that there are no non-trivial orientable boundary compressions for a \gi{} surface at a single essential boundary loop, since any arc on a torus with both ends on locally the same side of the loop $K \cap T_k$ can be isotoped onto $K$, thus describing a geometric compression.

Bartolini~\cite{boundarycomp} establishes various structural properties of bounded, one-sided surfaces in link spaces:

\begin{prop}[{\cite[Prop 3.1]{boundarycomp}}]\label{prop:uniquecollar}
Within a \textit{torus}$\times I$, any pair of inner and outer boundary slopes describes a unique geometrically incompressible one-sided surface, up to isotopy.
\end{prop}

\begin{lem}[{\cite[Lem 3.3]{boundarycomp}}]\label{lem:boundary_curves}
Given a bounded, \gi{}, boundary incompressible surface $K_0$ in a link space $M$, there is at most one \gi{}, boundary compressible surface $K$, with a given boundary slope $\sigma_k$ on $T_k$, that can be obtained by attaching M\"obius bands to $K_0$.
\end{lem}

\begin{defn}\label{defn:polygonal}
A disc $d \subset M$ for a bounded surface $K$ in a link space $M$ is polygonal if $d^{\circ}$ is embedded and $\partial d$ is partitioned into $2n$ arcs, $\lambda_1, \lambda_2, \ldots, \lambda_{2n}$ for $n \geq 2$, where $\lambda_i \cap \partial K = \partial \lambda_i$ for all $i$, and $\lambda_{2k-1} \subset K$ and $\lambda_{2k} \subset \partial M$ for $k=1, \ldots, n$. Each arc $\lambda_{2k-1}$ is essential in $K$ and each arc $\lambda_{2k}$ is embedded, disjoint from all other $\lambda_{2s}$, $s \not= k$ and not homotopic into $\partial K$ along $\partial M$ rel endpoints.
\end{defn}

\begin{prop}[{\cite[Prop 4.4]{boundarycomp}}]\label{prop:boundary_incomp}
Any bounded, \gi{} one-sided surface embedded in a link manifold that is at the collection $\{T_k\} \subseteq \partial M$ has a boundary incompressible restriction outside of a neighbourhood of $\{T_k\}$, which is unique up to isotopy if the manifold contains no non-trivial, embedded, irreducible, polygonal discs for the boundary incompressible surface at $\{T_k\}$.
\end{prop}

\section{Even Dehn fillings of Figure 8 knot space}\label{sect:fig8}

The $(2p, q)$ Dehn fillings of Figure 8 knot space, which are irreducible, hyperbolic, non-Haken $3$-manifolds (under modest restrictions of $p$ and $q$), have \gi{} \oss{} surfaces, as determined by the existence of non-zero $\mathbb{Z}_2$-homology. The \gi{}, boundary incompressible surfaces in the knot exterior alone have been known since Thurston~\cite{thurston}, providing a foundation for the study of the \gi{} surfaces in the filled manifolds. 

By a direct construction, it is possible to show that \gi{} one-sided surfaces in the filled manifolds are a strict subset of the surfaces that arise from completing the boundary incompressible surfaces in the knot space. Furthermore, the unique \gi{} splitting surface for any member of the class can be identified and is determined determined solely by the filling ratio $\frac{p}{q}$. This contrasts significantly with the dearth of information in the two-sided case; no classification exists for the strongly irreducible \tss{s} of any closed hyperbolic $3$-manifold.

\begin{thm}
Every $(2p, q)$ Dehn filling of Figure 8 knot space, where $p, q \in \mathbb{Z}$, $(2p, q) = 1$, $(2p, q) \not= (0,\pm 1), (2,\pm1), (4,\pm1)$, has a unique, \gi{} \oss{} surface, which can be identified precisely by the ratio $\frac{p}{q}$.
\end{thm}

\begin{proof}
Let $M_0$ be the Figure 8 knot exterior, $M_{(2p, q)}$ the $(2p, q)$ filling of $M_0$ and $M_T = \overline{M_{(2p, q)} \setminus M_0}$ the solid filling torus. Consider the torus $T = \partial M_0 = \partial M_T$, which can be parametrised by meridian, longitude co-ordinates with reference to either side. Denote knot space and torus co-ordinates as $(m, l)_K$ and $(l^{\prime}, m^{\prime})_T$ respectively. Let $m, l$ be the canonical meridian and longitude for $T$ with respect to the knot space.

Choose torus co-ordinates such that the meridian disc for $M_T$ with boundary $(0, 1)_T$ is glued along the $(2p, q)_K$ curve by the filling. Note the interchanged order of the meridian and longitude in the two systems. This corresponds to the difference in conventions for labelling knot space and torus co-ordinates, which is dictated by the curve that generates homology in each case. Let $\mathbf{A}$ be the transition matrix from torus to knot space co-ordinates, where $a, b \in \mathbb{Z}$ are chosen such that $det(\mathbf{A})= 1$ and the Euclidean distance of $(a, b)$ from $(0, 0)$ is strictly less than that of $(2p, q)$:

\begin{displaymath}
\mathbf{A} = 
\left( \begin{array}{cc}
b & 2p\\
a & q\\
\end{array} \right)
\end{displaymath}

If a \oss{} surface is \gi{}, it can be isotoped to be incompressible and boundary incompressible in $M_0$.  By Thurston~\cite{thurston}, there are three such surfaces in Figure 8 knot space: the Seifert surface for the knot $P_{(0, 1)}$, and the two punctured Klein bottle spanning surfaces $P_{(4, 1)}, P_{(4, -1)}$ with boundary curves $(4, 1)_K, (4, -1)_K$ respectively. Note that $P_{(0, 1)}$ is a punctured torus with boundary $(0, 1)_K$, over which the knot space $M_0$ fibres. As the Figure 8 knot is small, the incompressible spanning surfaces are known to be free, as such, the complements of $P_{(0, 1)}, P_{(4, 1)}, P_{(4, -1)}$ in $M_0$ are handlebodies.

While it is possible to have non-intersecting copies of $P_{(0, 1)}$ in $M_0$, disconnected surfaces in $M_0$ are disregarded here: by Rubinstein~\cite{rubin78}, there are no \gi{} one-sided surfaces with multiple boundary components in a solid torus. Therefore, the only possible connecting surfaces in $M_T$ are annuli, as a cyclic fundamental group is required. However, a non-orientable surface cannot arise by joining distinct orientable components.

Having ascertained the three possible knot space components of any \gi{} \oss{}, consider the potential closures in $M_T$. By Lemma~\ref{lem:boundary_curves}, a \gi{} surface in a solid torus is completely determined by its boundary slope. As such, the surfaces can each be explicitly identified using $\mathbf{A}^{-1}$ to change to torus co-ordinates:

\begin{displaymath}
\mathbf{A}^{-1} = 
\left( \begin{array}{cc}
q & -2p\\
-a & b\\
\end{array} \right)
\end{displaymath}

\begin{center}
$(0, 1)_K = (-2p, b)_T$, $(4, 1)_K = (4q-2p, -4a+b)_T$, $(4, -1)_K = (-4q-2p, 4a+b)_T$.
\end{center}

By way of example, calculate the curves for $M_{(8, 3)}$:

\begin{displaymath}
\mathbf{A}^{-1} = 
\left( \begin{array}{cc}
3 & -8\\
-1 & 3\\
\end{array} \right)
\end{displaymath}

\begin{center}
$(0, 1)_K = (-8, 3)_T$, $(4, 1)_K = (4, -1)_T$, $(4, -1)_K = (20, -7)_T$.
\end{center}

Take the unions of the \gi{} components in the torus and knot space to obtain the splitting surfaces $K_{(0, 1)}, K_{(4, 1)}, K_{(4, -1)}$ arising from the $(0, 1)_K,  (4, 1)_K, (4, -1)_K$ curves respectively. Notice that in the example, $K_{(0, 1)}$ and $K_{(4, 1)}$ both have non-orientable genus $4$, while $K_{(4, -1)}$ has genus $6$.

Consider whether the splitting surfaces thus obtained are \gi{} in the filled manifolds $M_{(2p, q)}$:

Given $M_{(2p, q)}$ is irreducible and non-Haken by the restriction to non-exceptional cases, any splitting surface compresses to a \gi{} one~\cite{rubin78}. Therefore, all the minimal genus surfaces are inherently \gi{}, as any compressions would imply a lower genus splitting. In the example, both $K_{(0, 1)}$ and $K_{(4, 1)}$ are \gi{}.

In order to determine which surfaces are minimal genus, consider the relative genera of the completions of each in the solid filling torus:

Begin by calculating the intersection numbers between the boundary curves:

\hspace{0.2\textwidth}\parbox[t]{0.6\textwidth}{$|(4q-2p, -4a+b)_T \cap (-2p, b)_T| = 4$\\
$|(-4q-2p, 4a+b)_T \cap (-2p, b)_T| = -4$\\
$|(4q-2p, -4a+b)_T \cap (-4q-2p, 4a+b)_T| = 8$.}

Surfaces in a solid torus are related by a single M\"obius band addition or boundary compression if and only if their boundary slopes have intersection number $\pm 2$. As such, the intersection numbers of the co-ordinates on $T$ show that, in $M_T$, the surfaces $K_{(4, 1)}, K_{(4, -1)}$ are related to $K_{(0, 1)}$ by a difference of two M\"obius bands. This may involve one boundary compression and one addition, or two boundary compressions (additions). Furthermore, the intermediate surface achieved after a single move has boundary curve $(2q - 2p, -2a +b)$ between $K_{(4, 1)}, K_{(0, 1)}$ or $(2q + 2p, -2a - b)$ between $K_{(4, -1)}, K_{(0, 1)}$.

Given the known relationship by boundary compressions, the moves from $K_{(0,1)}$ to $K_{(4,1)}$ and $K_{(0,1)}$ to $K_{(4,-1)}$ each form chains of three vertices joined by two edges in $\Gamma$. By Observation 1 of Bartolini~{\cite[p.5]{boundarycomp}}, if two surfaces bounded by $(2p, q), (2p^{\prime}, q^{\prime})$ are known to be related by a boundary compression, the relative genera can be determined by a comparison of $|p|, |p^{\prime}|$. Hence, a local comparison of genus can be established:

\begin{center}
\begin{tabular}{|c|c|}
\hline
Filling ratio & Relative size of longitudinal co-ordinates \\
\hline
\hspace*{9mm} $\frac{p}{q} < -\frac{3}{2}$ & $|{-4q}{-2p}| < |2q + 2p| < |{-2p}| < |2q{-2p}| < |4q{-2p}|$\\
\hline
$-\frac{3}{2} < \frac{p}{q} < -\frac{1}{2}$ & $|2q + 2p| < |{-2p}| < |2q{-2p}| < |4q{-2p}|$\\
& $|2q + 2p| < |-{4q}{-2p}|$\\
\hline
$-\frac{1}{2} < \frac{p}{q} < \frac{1}{2}$\hspace*{3mm} & $|{-2p}| < |2q{-2p}| < |4q{-2p}|$\\
& $|{-2p}| < |2q + 2p| < |{-4q}{-2p}|$\\
\hline
$\frac{1}{2} < \frac{p}{q} < \frac{3}{2}$ & $|2q{-2p}| < |{-2p}| < |2q+2p| < |{-4q}{-2p}|$\\
& $|2q{-2p}| < |4q{-2p}|$\\
\hline
$\frac{3}{2} < \frac{p}{q}$\hspace*{8mm} & $|4q{-2p}| < |2q{-2p}| < |{-2p}| < |2q+2p| < |{-4q}{-2p}|$\\
\hline
\end{tabular}
\end{center}

Note that the partition of behaviour according to size of the ratio $\frac{p}{q}$ mirrors the Observations on branching found in the M\"obius band tree in Bartolini~\cite{boundarycomp}.

This data can be plotted schematically in the M\"obius band tree to get a visual representation of the relative genera of the surfaces (see Figures~\ref{fig:tree_difference_2}, \ref{fig:tree_difference_1}, \ref{fig:tree_difference_3}):

\begin{figure}[h]
   \footnotesize
   \centering
   \includegraphics[width=4.5in]{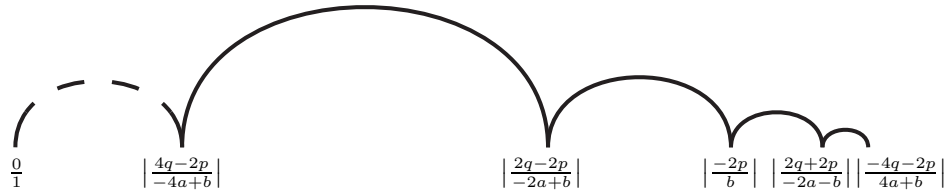}
   \put(-327,-12){$\frac{0}{1}$}
   \put(-277,-12){$\big|\frac{4q{-2p}}{{-4a}+b}\big|$}
   \put(-141,-12){$\big|\frac{2q{-2p}}{{-2a}+b}\big|$}
   \put(-65,-12){$\big|\frac{{-2p}}{b}\big|$}
   \put(-39,-12){$\big|\frac{2q+2p}{{-2a}{-b}}\big|$}
   \put(-7,-12){$\big|\frac{{-4q}{-2p}}{4a+b}\big|$}
   \caption{Relative position in $\Gamma$ for $\frac{p}{q} > \frac{3}{2}$ }
   \label{fig:tree_difference_2}
   \vspace{2mm}
\end{figure}

\begin{figure}[h]
   \footnotesize
   \centering   
   \includegraphics[width=4.5in]{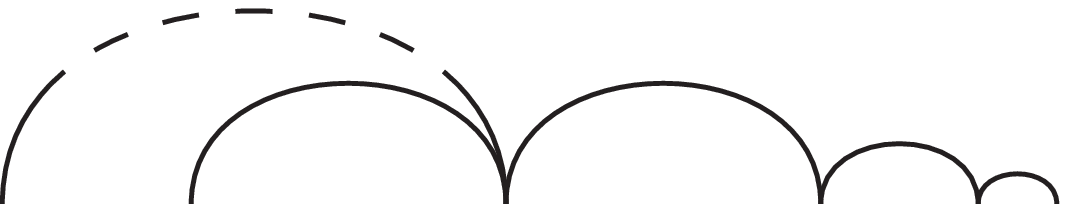}
   \put(-327,-12){$\frac{0}{1}$}
   \put(-283,-12){$\big|\frac{4q{-2p}}{{-4a}+b}\big|$}
   \put(-186,-12){$\big|\frac{2q{-2p}}{{-2a}+b}\big|$}
   \put(-85,-12){$\big|\frac{{-2p}}{b}\big|$}
   \put(-42,-12){$\big|\frac{2q+2p}{{-2a}{-b}}\big|$}
   \put(-10,-12){$\big|\frac{{-4q}{-2p}}{4a+b}\big|$}
   \caption{Relative position in $\Gamma$ for $\frac{1}{2} < \frac{p}{q} < \frac{3}{2}$ }
   \label{fig:tree_difference_1}
\end{figure}

\begin{figure}[h]
   \footnotesize
   \centering  
   \vspace{2mm}
   \includegraphics[width=4.5in]{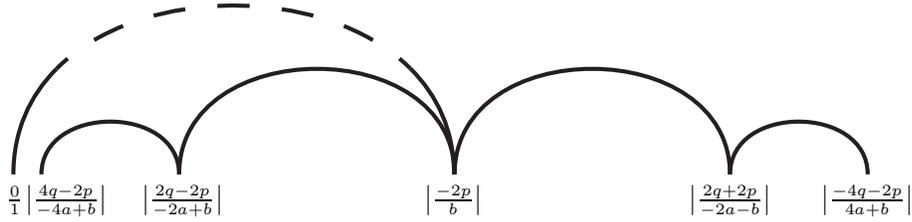}
   \put(-327,-12){$\frac{0}{1}$}
   \put(-320,-12){$\big|\frac{4q{-2p}}{{-4a}+b}\big|$}
   \put(-276,-12){$\big|\frac{2q{-2p}}{{-2a}+b}\big|$}
   \put(-169,-12){$\big|\frac{{-2p}}{b}\big|$}
   \put(-69,-12){$\big|\frac{2q+2p}{{-2a}{-b}}\big|$}
   \put(-19,-12){$\big|\frac{{-4q}{-2p}}{4a+b}\big|$}
   \caption{Relative position in $\Gamma$ for $-\frac{1}{2} < \frac{p}{q} < \frac{1}{2}$}
   \label{fig:tree_difference_3}
\end{figure}

The cases $\frac{p}{q} < -\frac{3}{2}$ and $-\frac{3}{2} < \frac{p}{q} < -\frac{1}{2}$ are symmetric to their positive counterparts, where $(4,-1)_K$ replaces $(4,1)_K$ and $(2q + 2p, -2a - b)_K$ replaces $(2q - 2p, -2a +b)_K$ throughout in the negative case.

The relative distance in $\Gamma$ shows that the minimal genus surfaces in the various cases are as follows:

\begin{center}
\begin{tabular}{|c|c|}
\hline
Filling ratio & Minimal genus surface(s) \\
\hline
\hspace*{9mm} $\frac{p}{q} < -\frac{3}{2}$ & $K_{(4,-1)}$\\
\hline
$-\frac{3}{2} < \frac{p}{q} < -\frac{1}{2}$ & $K_{(0, 1)}, K_{(4, -1)}$\\
\hline
$-\frac{1}{2} < \frac{p}{q} < \frac{1}{2}$\hspace*{3mm} & $K_{(0, 1)}$\\
\hline
$\frac{1}{2} < \frac{p}{q} < \frac{3}{2}$ & $K_{(0, 1)}, K_{(4, 1)}$\\
\hline
$\frac{3}{2} < \frac{p}{q}$\hspace*{8mm} & $K_{(4, 1)}$\\
\hline
\end{tabular}
\end{center}

Therefore, for all $p, q$, $(2p,q)=1$, $(2p, q) \not= (0,\pm 1), (2,\pm1), (4,\pm1)$, there is a unique, \gi{}, minimal genus, \oss{} surface in $M_{(2p, q)}$ when $\big|\frac{p}{q}\big| > \frac{3}{2}$ or $0 < \big|\frac{p}{q}\big| < \frac{1}{2}$; and, for $\frac{1}{2} < \big|\frac{p}{q}\big| < \frac{3}{2}$, there are two \gi{}, minimal genus \oss{s}, where the relative sign of $p,q$ identifies the surface(s) in all cases.

\begin{claim}
If $\frac{p}{q} > -\frac{1}{2}$, then $K_{(4, -1)}$ compresses to $K_{(0, 1)}$.
\end{claim}

By the preceding determination of relative surface genera under different filling ratios, for all $\frac{p}{q} > -\frac{1}{2}$, the surface $K_{(4,-1)}$ is not of minimal genus and differs from $K_{(0,1)}$ by two additional M\"obius bands in the filling torus $M_T$ (refer to Figures~\ref{fig:tree_difference_2}, \ref{fig:tree_difference_1}, \ref{fig:tree_difference_3}.)

Consider the surfaces $P_{(0, 1)}, P_{(4, -1)}$ in the knot space. Given Figure 8 knot space is a fibre bundle over its Seifert surface, the punctured torus $P_{(0, 1)}$, consider $P_{(4, -1)}$ in relation to this fibre structure. The boundary curves $\partial P_{(0, 1)}, \partial P_{(4, -1)}$ intersect in four points $x_1, x_2, x_3, x_4$, which divide $\partial P_{(0, 1)}$ into four segments: $c_1, c_2, c_3, c_4$.

By Thurston~\cite{thurston}, $P_{(4, -1)}$ meets $P_{(0, 1)}$ in $\sigma_1, \sigma_2$, which are disjoint parallel arcs that each run once longitudinally around $P_{(0, 1)}$. These arcs have endpoints $\{x_1, x_4\}, \{x_2, x_3\}$ respectively, which bound disjoint segments of $\partial P_{(0, 1)}$ \--- say $c_1, c_3$. See Figure~\ref{fig:P_3}.

\begin{figure}[h]
   \centering
   \includegraphics[width=3in]{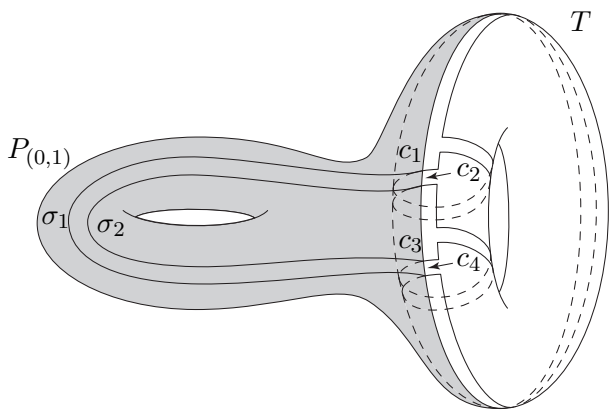}
   \put(-15, 143){$T$}
   \put(-228, 95){$P_{(0, 1)}$}
   \put(-215, 70){$\sigma_1$}
   \put(-194, 67){$\sigma_2$}
   \put(-80, 95){$c_1$}
   \put(-58, 88){$c_2$}
   \put(-80, 61){$c_3$}
   \put(-58, 54){$c_4$}
   \caption{$P_{(4, -1)}$ intersects $P_{(0, 1)}$ in the arcs $\sigma_1, \sigma_2$ and $T$ in the curve $(4, -1)_K$.}
   \label{fig:P_3}
\end{figure}

Since $K_{(4, -1)}$ differs from $K_{(0, 1)}$ in $M_T$ by two additional M\"obius bands, push these bands out of $M_T$, so that the surfaces coincide in the solid filling torus and $K_{(4, -1)}$ intersects $T$ in the $(-2p, q)_T = (0, 1)_K$ curve. By Lemma~\ref{lem:boundary_curves}, there is only one surface $\bar{P}_{(4, -1)} \subset M_0$ with boundary curve $(0, 1)_K$ that boundary compresses to $P_{(4, -1)}$, which is obtained by adding two parallel M\"obius bands to $P_{(4, -1)}$.

If $b_1, b_2 \subset \bar{P}_{(4, -1)}$ are the M\"obius bands of $K_{(4, -1)}$ pushed out of $M_T$, let $\bar{c}_2, \bar{c}_4$ be their intersections with $P_{(0, 1)}$.  Let $\bar{\sigma}_1, \bar{\sigma}_2$ be the remaining sections of $\sigma_1, \sigma_2$ in $\bar{P}_{(4, -1)}$. The disc $D_1 \subset M_0$ bounded by $\bar{\sigma}_1, \bar{\sigma}_2, \bar{c}_2, \bar{c}_4$ is a compressing disc for $\bar{P}_{(4, -1)}$

Compressing $\bar{P}_{(4, -1)}$ along $D_1$ results in a punctured surface in $M_0$, with boundary $(0, -1)_K = (0, 1)_K$, that does not intersect $P_{(0, 1)}$. As this surface shares a boundary slope with, yet is disjoint from, the orbit surface, it compresses to a surface isotopic to $P_{(0, 1)}$.

Having thus proved the claim, the result can be extended by symmetry to the case where $\frac{p}{q} < \frac{1}{2}$. Therefore, for $-\frac{1}{2} < \frac{p}{q} < \frac{1}{2}$, the surface $K_{(0,1)}$ is the unique \gi{} surface in $M_{(2p,q)}$ and for $\big|\frac{p}{q}\big| > \frac{1}{2}$, there are at most two \gi{} \oshs{} surfaces.

In the event that $K_{(0, 1)}$ is also of non-minimal genus, the result can be sharpened:

\begin{claim}
If $\frac{p}{q} > \frac{3}{2}$, then $K_{(0, 1)}$ compresses to $K_{(4, 1)}$.
\end{claim}

Since $K_{(0, 1)}$ has non-minimal genus, it differs from $K_{(4, 1)}$ in $M_T$ by two additional M\"obius bands. Push these bands out of $M_T$, so that the surfaces coincide in the solid filling torus and $K_{(0, 1)}$ intersects $T$ in the $(4q-2p, 4a+q)_T = (4, 1)_K$ curve. By Lemma~\ref{lem:boundary_curves}, there is only one surface $\bar{P}_{(0, 1)} \subset M_0$ with boundary curve $(4, 1)_K$ that boundary compresses to $P_{(0, 1)}$, which is obtained by adding two parallel M\"obius bands to $P_{(0, 1)}$.

Let $\alpha, \beta \subset \bar{P}_{(0, 1)}$ be arcs running along the M\"obius bands of $K_{(0, 1)}$ pushed out of $M_T$, with endpoints $\{a_1, a_2\}, \{b_1, b_2\} \in \partial \bar{P}_{(0, 1)}$ respectively. Let $\lambda$ be an arc that runs once longitudinally around $\bar{P}_{(0, 1)}$, with endpoints $\{a_1, b_1\}$. Let $\gamma$ be an arc that runs once meridionally around $\bar{P}_{(0, 1)}$, with endpoints $\{a_2, b_2\}$. See Figure~\ref{fig:P_2}.

The disc $D_2 \subset M_0$ bounded by $\alpha, \beta, \gamma, \lambda$ is a compressing disc for $\bar{P}_{(0, 1)}$.

\begin{figure}[h]
   \centering
   \includegraphics[width=3in]{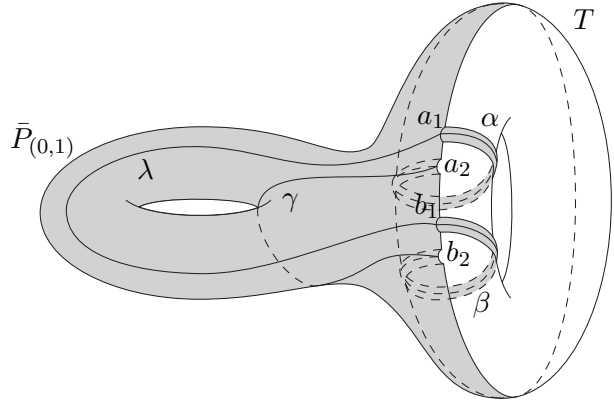}
   \put(-15, 140){$T$}
   \put(-228, 95){$\bar{P}_{(0, 1)}$}
   \put(-180, 83){$\lambda$}
   \put(-125, 73){$\gamma$}
   \put(-50, 103){$\alpha$}
   \put(-74, 104){$a_1$}
   \put(-64, 86){$a_2$}
   \put(-53, 33){$\beta$}
   \put(-75, 70){$b_1$}
   \put(-63, 52){$b_2$}
   \caption{The arcs $\alpha, \beta, \gamma, \lambda$ on $\bar{P}_{(0, 1)}$ that bound the compressing disc $D_2 \subset M_0$.}
   \label{fig:P_2}
\end{figure}

Compressing along $D_2$ results in a once-punctured Klein bottle in $M_0$, with boundary slope $(4, 1)_K$. By Thurston~\cite{thurston}, such a surface is unique up to isotopy, so the resulting surface is isotopic to $P_{(4, 1)}$. Therefore, $K_{(0, 1)}$ geometrically compresses to $K_{(4, 1)}$.

Again, symmetry shows that the corresponding result holds in the case where $\frac{p}{q} < -\frac{3}{2}$. Therefore, given a $(2p, q)$ Dehn filling of Figure 8 knot space with ratio $\frac{p}{q} > \frac{3}{2}$, the surface $K_{(4, 1)}$ is the unique \gi{} \oshs{} surface for the manifold, up to isotopy; and, for $\frac{p}{q} < -\frac{3}{2}$ such a surface is $K_{(4, -1)}$.

In the case where $\frac{1}{2} < \big|\frac{p}{q}\big| < \frac{3}{2}$, it has been established that there are two surfaces of minimal genus. In order to determine whether such surfaces are isotopic, consider the known structure of Figure 8 knot space:

There exists an embedded quadrilateral disc $D_3$ for $P_{(0,1)}$, with boundary running once longitudinally and once meridionally around the fibre torus, and two boundary curves on the manifold boundary, each running once meridionally around $T$. This can be seen in the construction for Figure~\ref{fig:P_2} above, where the curves $\alpha, \beta$ instead lie on $T$ for the original boundary incompressible surface. Notice that there is a dual symmetric disc $D_4$ with negative meridional slopes on $T$. Both discs $D_3, D_4$ can be visualised most readily in the twisted ribbon picture of the Seifert surface for the Figure 8 knot (see Figure~\ref{fig:strips}).

\begin{figure}[h]
   \centering
   \includegraphics[width=1.5in]{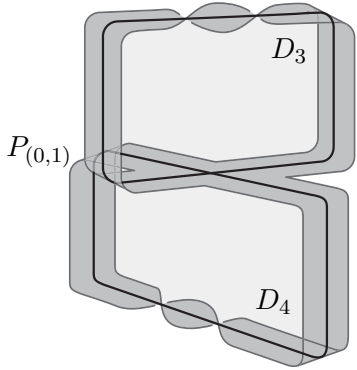}
   \put(-32, 117){$D_3$}
   \put(-132, 80){$P_{(0,1)}$}
   \put(-38, 22){$D_4$}
   \caption{The embedded quadrilateral discs $D_3, D_4$ as seen in the ribbon picture of $P_{(0,1)}$.}
   \label{fig:strips}
\end{figure}

In the case where $\frac{1}{2} < \frac{p}{q} < \frac{3}{2}$, it is known that $K_{(0,1)}, K_{(4,1)}$ have the same minimal genus. Push a single M\"obius band of $K_{(0,1)}$ out of $M_T$, to obtain the surface $\hat{P}_{(0,1)}$ with boundary slope $(2,1)_K$. The M\"obius band thus introduced to the knot space may be immediately boundary compressed back into $M_T$ via an embedded bigon $B$ to recover $P_{(0,1)}$. However, the introduction of this band matches exactly one of the boundary arcs of the quadrilateral disc $D_3$, thus this disc has only one boundary curves on $T$ for $\hat{P}_{(0,1)}$. Therefore, $D_3$ is an embedded bigon for $\hat{P}_{(0,1)}$, which is dual to the bigon $B$. See Figure~\ref{fig:isotopy}.

\begin{figure}[h]
   \centering
   \includegraphics[width=3in]{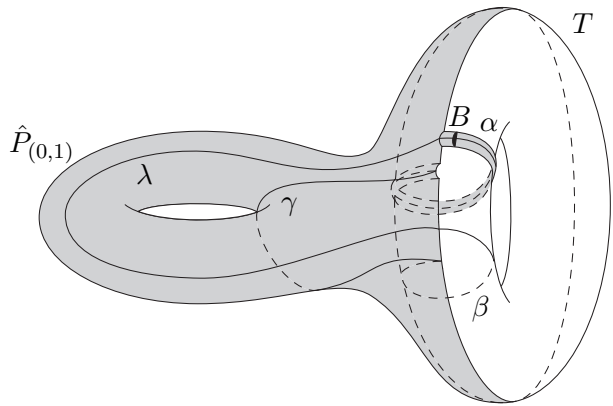}
   \put(-15, 140){$T$}
   \put(-228, 95){$\hat{P}_{(0, 1)}$}
   \put(-180, 83){$\lambda$}
   \put(-125, 73){$\gamma$}
   \put(-50, 103){$\alpha$}
   \put(-62, 105){$B$}
   \put(-53, 33){$\beta$}
   \caption{The arcs $\alpha, \gamma, \lambda \subset \hat{P}_{(0, 1)}$ and $\beta \subset T$ that bound the embedded bigon $D_3 \subset M_0$ dual to $B$.}
   \label{fig:isotopy}
\end{figure}

Boundary compressing $\hat{P}_{(0,1)}$ along the bigon $D_3$ results in a surface $\hat{P}_{(0,1)}^{\prime}$ with boundary $(4,1)_K$. If this surface were boundary compressible, it would boundary compress to one of $P_{(0,1)}$ or $P_{(4,1)}$ by minimality of genus. However, within the solid filling torus, the move in the knot space from $P_{(0,1)}$ to $\hat{P}_{(0,1)}^{\prime}$ has involved a subtraction followed by an addition of a M\"obius band to $K_{(0,1)} \cap M_T$. If $\hat{P}_{(0,1)}^{\prime}$ were to boundary compress, terminating by obtaining either of $P_{(0,1)}$ or $P_{(4,1)}$, such compressions would add further M\"obius bands to $K_{(0,1)} \cap M_T$, contradicting minimality of genus. As such, $\hat{P}_{(0,1)}^{\prime}$ is boundary incompressible and isotopic to $P_{(4,1)}$ in $M_0$. Therefore, for $\frac{1}{2} < \frac{p}{q} < \frac{3}{2}$, the surfaces $K_{(0,1)}, K_{(4,1)}$ are isotopic.

By symmetry, a similar pair of moves can be used to isotope $K_{(0,1)}$ along $D_4$ to get $K_{(4,-1)}$, showing that for $-\frac{3}{2} < \frac{p}{q} < -\frac{1}{2}$, the surfaces $K_{(0,1)}, K_{(4,-1)}$ are isotopic.

Therefore, for the class where $\frac{1}{2} < \big|\frac{p}{q}\big| < \frac{3}{2}$, there is also a unique \gi{} surface, which is $K_{(0,1)}$ regardless of relative sign.
\end{proof}

Note that the final case addressed in this argument demonstrates that a \gi{} surface in a filled knot space may have non-isotopic boundary incompressible restrictions in the knot space. However, the Figure 8 knot exterior contains embedded quadrilateral discs for its boundary incompressible surfaces and, as demonstrated, one such disc is key to interpolating between the distinct restrictions. As shown in Proposition~\ref{prop:boundary_incomp}, embedded irreducibly polygonal discs in link spaces are related to movement of M\"obius bands between boundary components. Therefore, it is conjectured that non-isotopic restrictions are not generic and are related to the combinatorics of such polygonal discs. This is the topic of ongoing work.

\bibliographystyle{abbrv}
\bibliography{oshs}

\textit{Department of Mathematics\\
Oklahoma State University\\
Stillwater, Oklahoma 74078\\
United States}

Email: \texttt{bartolini@math.okstate.edu}

\end{document}